\documentclass[12pt,reqno]{amsart}
\usepackage{mathrsfs}
\usepackage{amsgen}
\usepackage{amscd}
\usepackage{amsmath}
\usepackage{latexsym}
\usepackage{amsfonts}
\usepackage{amssymb}
\usepackage{amsthm}
\usepackage{graphicx}
\usepackage{color}
\usepackage{amsthm,amssymb}
\usepackage{graphicx}
\usepackage{enumerate}
\usepackage{amssymb,amsmath,graphicx,amsfonts,euscript}
\usepackage{color}
\usepackage{mathrsfs}
\usepackage{empheq}
\usepackage{cases}
\usepackage{cite}

\usepackage[margin=3cm, a4paper]{geometry}

\def\H{{\cal H}}

\def\R{\mathbb{R}}

\def\C{\mathbb{C}}

\def\H2{H^2(\R^N)}
\def\L2{L^2(\R^N)}

\def\H{{\cal H}}

\def\H1{H^1(\R)}

 \newcommand{\de}{\delta}

\newcommand{\supp}{\operatorname{supp}}

 \newcommand{\Del}[1]{}

\numberwithin{equation}{section}

\newtheorem{thm}{Theorem}[section]

\newtheorem{lem}[thm]{Lemma}

\theoremstyle{remark}

\newtheorem*{exam*}{Examples}

\begin{document}

\setcounter{page}{1}

\title[BLOW-UP]{Finite time/Infinite time blow-up behaviors for the inhomogeneous nonlinear Schr\"{o}dinger equation}

\author{Ruobing Bai}
\address{Center for Applied Mathematics\\
Tianjin University\\
Tianjin 300072, China}
\email{baimaths@hotmail.com}
\thanks{}

\author{Bing Li}
\address{School of Mathematical Sciences\\Tiangong University\\
Tianjin 300387, China}
\email{binglimath@gmail.com}
\thanks{}


\subjclass[2010]{Primary  35Q55; Secondary 35B44}


\keywords{inhomogeneous nonlinear Schr\"{o}dinger equation, blow-up,  localized virial identity}

\maketitle

\begin{abstract}\noindent
In this work, we consider the following focusing inhomogeneous nonlinear Schr\"odinger equation \begin{align*}
 i\partial_t u+\Delta u +|x|^{-b}|u|^p u=0,\quad (t, x)\in\R\times\R^N
\end{align*}
with $0<b<\mbox{min}\{2, N\}$ and $\frac{4-2b}{N}<p<\frac{4-2b}{N-2}$. Assume that $u_0 \in H^{1}(\R^N)$ and beyond the ground state threshold,
then we prove the following two statements,

(1) when $\frac{4-2b}{N}<p< \min\{\frac{4}{N}, \frac{4-2b}{N-2}\}$, or $p =\frac{4}{N}$ when $b \in (0, \frac 4 N)$,
 then the corresponding solution blows up in finite time;

(2) when $\frac{4}{N}<p<\frac{4-2b}{N-2}$,  we prove the finite or infinite time blow-up. Moreover, we can further obtain a precise lower bound of infinite time  blow-up rate, that is
\begin{equation*}
  \sup_{t\in[0,T]}\|\nabla u(t)\|_{L^2}\gtrsim T^{\kappa},\quad \mbox{for some} \quad \kappa>0.
\end{equation*}

To our knowledge, the statement (1) establishes the first finite time blow-up result for this equation in the intercritical case when the initial data $u_0$ doesn't have finite variance and is non-radial. The statement (2) gives the first result for the infinite time blow-up rate for this equation.
\end{abstract}

\section{Introduction}
\vskip 0.2cm
In this paper, we consider the following focusing inhomogeneous nonlinear Schr\"odinger (INLS) equation
\begin{equation}\label{IDNLS}
   \left\{ \aligned
    &i\partial_t u+\Delta u +|x|^{-b}|u|^p u=0,
    \\
    &u(x,0)=u_0,
   \endaligned
  \right.
 \end{equation}
where $u=u(t,x): \R\times\R^N\rightarrow \C$ is a complex-valued function, $0<b<\min\{2, N\}$ and $\frac{4-2b}{N}<p<2^{\ast}_b$, where $2^{\ast}_b=\infty$ if $N=1,2$, and $2^{\ast}_b=\frac{4-2b}{N-2}$ if $N\geq3$.
The equation has various physical contexts such as stable high power propagation can be achieved
in plasma by sending a preliminary laser beam that creates a channel with a reduced electron density, and thus reduces the nonlinearity inside the channel (see \cite{Gill-00, Liu-Tripathi-94}).
In particular, if $b=0$, equation \eqref{IDNLS} is the following classical nonlinear Schr\"odinger (NLS) equation
\begin{equation}\label{1.2222222}
   i\partial_t u+\Delta u +|u|^p u=0.
 \end{equation}

It is well known that if $u(t,x)$ is a solution of \eqref{IDNLS}, so is
\[
u_{\lambda}(t,x)=\lambda^{\frac{2-b}{p}}u(\lambda^2t,\lambda x), \quad \lambda>0.
\]
Denote $s_c=\frac{N}{2}-\frac{2-b}{p}$, then the scaling leaves $\dot{H}^{s_c}$ norm invariant, that is,
\[
\|u\|_{\dot{H}^{s_c}}=
\|u_{\lambda}\|_{\dot{H}^{s_c}}.
\]
Then, $s_c$ is called the $critical$ $regularity$. If $s_c = 0$ ($p=\frac{4-2b}{N}$), the problem is known as mass-critical or $L^2$-critical;
if $s_c < 0$ ($0<p<\frac{4-2b}{N}$), it is called mass-subcritical or $L^2$-subcritical; if $0 < s_c < 1$, it is known as mass-supercritical and energy-subcritical (or intercritical). In addition, the solution satisfies
the conservation of mass and energy, defined respectively by
\begin{align*}
&M\big(u(t)\big):=\int_{\R^N}|u|^2dx=M(u_0),\\
&E\big(u(t)\big):=\frac12\int_{\R^N}|\nabla u|^2dx-\frac{1}{p+2}\int_{\R^N}|x|^{-b}|u|^{p+2}dx=E(u_0).
\end{align*}

The well-posedness theory of equation \eqref{IDNLS} was first studied by  Genoud-Stuart \cite{Genoud-Stuart-08}. Using the abstract theory of Cazenave \cite{Cazenave-03}, they showed that the equation \eqref{IDNLS} is local well-posedness in $H^1(\R^N)$ if $0<p<2^{\ast}_b$ ($s_c<1$), as well as global well-posedness for any initial data if $0<p<\frac{4-2b}{N}$ ($s_c<0$) and sufficiently small initial data if $\frac{4-2b}{N}\leq p<2^{\ast}_b$ ($0\leq s_c<1$) in $H^1(\R^N)$. Farah\cite{Farah-16} proved a
Gagliardo-Nirenberg type estimate and used it to establish sufficient conditions for global well-posedness  in $H^{1}(\R^N)$.
Guzm\'{a}n \cite{Guzman-17} used the contraction mapping principle based on the known classical Strichartz estimates to establish the local and global well-posedness for
the INLS equation \eqref{IDNLS} in Sobolev spaces $H^s(\R^N), 0\leq s \leq1$. As for the scattering theory of equation \eqref{IDNLS}, Farah and Guzm\'{a}n  \cite{Farah-Guzman-171} proved that when the initial data is radial, $0<b<\frac12$, and $p=2$, then the corresponding solution  scatters in $H^1(\R^3)$. Later, Farah and Guzm\'{a}n \cite{Farah-Guzman-172} extended the range of $b$ and $p$, and obtained the energy scattering to all spacial dimensions $N\geq2$.
Some other results related to the well-posedness and scattering can be found in \cite{Dinh-19, Dinh-20, Dinh-21,Xu-Zhao-19, Cardoso-Farah-Guzman-Murphy-20, Campos-21, Dinh-Keraani-21} and the references therein.



Furthermore, when $p=\frac{4-2b}{N}$, $u_0 \in H^1(\R^N)$ and $\|u_0\|_{L^2} < \|Q\|_{L^2}$, Genoud \cite{Genoud-12} showed that equation \eqref{IDNLS}
is globally well-posed in $H^1(\R^N)$, and proved the existence of critical mass blow-up solutions. Here $Q$ is the ground state of the elliptic equation
\begin{align}\label{3.241}
-Q+\triangle Q+|x|^{-b}|Q|^{p}Q=0.
\end{align}
We give a brief introduction of the existence and uniqueness of ground state solutions.  The existence of the ground state was obtained by Genoud-Stuart \cite{Genoud-08, Genoud-Stuart-08} in dimensions $N\geq2$, and
by Genoud \cite{Genoud-10} in dimension $N = 1$. Uniqueness was obtained in dimensions $N \geq 3$ by Yanagida \cite{Yanagida-91} (see
also Genoud \cite{Genoud-08}), in dimensions $N = 2$ by Genoud \cite{Genoud-11} and in dimension $N = 1$ by Toland \cite{Toland-84}.

In recent years, the blow-up theory of equation \eqref{IDNLS} has also been extensively studied. Farah \cite{Farah-16} proved that if the initial data  $u_0\in \{f\in H^{1}(\R^N); |x|f\in L^2(\R^N )\}$ and satisfies
\begin{align}
  E(u_0)^{s_c}M(u_0)^{1-s_c}<E(Q)^{s_c}M(Q)^{1-s_c},\label{1.4}
\end{align}
and
\begin{align}
  \|\nabla u_0\|_{L^2}^{s_c}\|u_0\|_{L^2}^{1-s_c}>\|\nabla Q\|_{L^2}^{s_c}\|Q\|_{L^2}^{1-s_c}\label{1.5}
\end{align}
with $0<s_c<1$, then the corresponding solution  blows up in finite time. In the intercritical case, Dinh \cite{Dinh-17} showed the finite time blow-up for radial initial
data $u_0\in H^1(\R^N)$ with negative energy or below ground state in dimensions $N \geq 2$. Campos and Cardoso \cite{Campos-Cardoso-20} established a scattering and blow-up dichotomy above the mass-energy threshold, which extended  the condition \eqref{1.4}. Later, Cardoso and Farah \cite{Cardoso-Farah-20} proved the finite time blow-up if the initial data $u_0$ is radial and $E(u_0)\leq 0$ in dimensions $N\geq 3$ with $0< s_c<1$.
 Ardila and Cardoso \cite {Ardila-Cardoso-21} proved that if $u_0\in H^{1}(\R^N)$ satisfies \eqref{1.4} and \eqref{1.5} with $0<s_c<1$, then the corresponding solution blows up in finite or infinite time.
Recently, Dinh and Keraani \cite{Dinh-Keraani-21} obtained a new blow-up criterion  to equation \eqref{IDNLS} and when $u(t)$  has finite variance then it blows up in finite time.
See also for examples some other  blow-up results in \cite{Merle-96, Chen-Guo-07, Raphael-Szeftel-11, zhu-14} and references therein for $b<0$ and more general nonlinear terms.

Criteria for the existence of finite-time blow-up solutions for the INLS equation \eqref{IDNLS} are known only in cases where $|x|u_0\in L^2(\R^N )$ or $u_0$ is radial. In this work, we aim to study the finite time blow-up theory for the INLS equation \eqref{IDNLS} for the non-radial data without finite variance. We prove that the solution for INLS equation \eqref{IDNLS} blows up in finite time when $\frac{4-2b}{N}<p< \min\{\frac{4}{N}, 2^{\ast}_b\}$ or $p =\frac{4}{N}$. The main blow-up mechanism is that for an ODE equation $f^l<f'$, if $l>1$, then $f$ will blow up in finite time. Besides, we give a lower bound for infinite time blow-up rate when $\frac{4}{N}<p<2^{\ast}_b$. More precisely, we prove the followings.

\begin{thm}\label{thm:main1}
Let $s_c=\frac{N}{2}-\frac{2-b}{p}$ and $0<b<\min\{2, N\}$. Assume that $u_0\in H^1(\R^N)$ and satisfies \eqref{1.4} and \eqref{1.5}. Let $u(t)$ be the solution of \eqref{IDNLS} defined in the maximal time interval of existence, say $I$. If one of the following cases holds,

 (1) $\frac{4-2b}{N}<p< \min\{\frac{4}{N}, 2^{\ast}_b\}$, or

(2) $p =\frac{4}{N}$, $0< b< \frac 4 N$,\\
then $I$ is finite.
\end{thm}

\begin{thm}\label{thm:main2}
Let $s_c=\frac{N}{2}-\frac{2-b}{p}$ and $0<b<\frac 4 N$. Assume that $u_0\in H^1(\R^N)$ and satisfies \eqref{1.4} and \eqref{1.5}. Let $u(t)$ be the solution of \eqref{IDNLS} defined in the maximal time interval of existence, say $I$.
If $\frac{4}{N}<p<2^{\ast}_b$,  then either $I$ is finite, or $I=\R$, and for any $T>0$,
\begin{align}\label{1.29348}
  \sup_{t\in[0,T]}\|\nabla u(t)\|_{L^2}\gtrsim T^{\frac{b}{2\alpha-2}},    \quad \quad with \quad \alpha=\frac{Np}{4}.
\end{align}

\end{thm}


To our best knowledge, this is the first finite time blow-up result for the equation \eqref{IDNLS} in the intercritical case when $u_0$ satisfies the conditions \eqref{1.4} and \eqref{1.5}, without the condition that $|x|u_0\in L^2(\R^N )$ or $u_0$ is radial.

For the classical nonlinear Schr\"odinger equation \eqref{1.2222222}, in the case of general $H^1$ data (not necessarily finite variance or radially
symmetric), the first blow-up result was proved by Glangetas and Merle \cite{Glangetas-Merle-95}.
They proved that negative energy solutions either blows up in finite time or blows up in infinite time in the sense that
\begin{align}\label{1.442434}
\sup_{t\in\R}\|\nabla u(t)\|_{L^2}=+\infty.
\end{align}
In \cite{Holmer-Roudenko-10}, Holmer and Roudenko showed that under the conditions that \eqref{1.4} and \eqref{1.5}, then the
corresponding solution to the focusing 3D cubic NLS either blows up in finite
time or blows up in infinite time.
Later, Du, Wu, and Zhang \cite{Du-Wu-Zhang-16} gave an alternative simple proof for this result
and extended it to all dimensions.
So, it is important to note that for the classical NLS, it is still unknown that $H^1$ solutions having negative energy or satisfying \eqref{1.4} and \eqref{1.5}
blow up in finite time (except the 1D mass-critical NLS \cite{Ogawa-Tsutsumi-91-2}).
The results of blow-up in some other situations can be referred to \cite{Glassey-77, Ogawa-Tsutsumi-91-1, Martel-97, Merle-Raphael-05, Zakharov-72} and references therein.

Compared with this, for the inhomogeneous nonlinear Schr\"odinger equation \eqref{IDNLS}, when $0<s_c<\min \{\frac{bN}{4}, 1\}$ or $s_c=\frac{bN}{4}$, we prove the finite time blow-up, and when $\frac{bN}{4}<s_c<1$, prove the finite or infinite time blow-up and obtain a lower bound of blow-up rate \eqref{1.29348}, which is finer than \eqref{1.442434}. The key to our proof is the use of the localised virial identity. Thanks to the existence of $|x|^{-b}$, we can obtain the good remainder estimates by using the Gagliardo-Nirenberg inequality when $\frac{4-2b}{N}<p\leq\frac4N$.
When $\frac{4}{N}<p< 2^{\ast}_b$, the above argument fails, so we choose a time-dependent radius in the localised virial identity. In particular, we take
\begin{align*}
R=R(T)=C\sup_{t\in [0,T]}\|\nabla u(t)\|_{L^2}^{\frac{2\alpha-2}{b}}.
 \end{align*}
This allows us to use the radius to control the blow-up time through delicated analysis, and thus gives the lower bound for infinite time blow-up rate.

This paper is organized as follows. In Section \ref{sec:notations}, we give some basic notation and identity. In Section \ref{sec:proof}, we give the proofs of Theorem \ref{thm:main1} and Theorem \ref{thm:main2}.
\vskip 0.2cm

\section{Notation and Local Virial identity}\label{sec:notations}

\subsection{Notation}

We write $X\lesssim Y$ or $Y\gtrsim X$ to denote the estimate $X\leq CY$ for some constant (which may depend on $p$, $b$, $N$ and the mass $M(u_0)$). Throughout the whole paper, the letter $C$ will denote different positive constants which are not important in our analysis and may vary line by line.

Before stating this theorem, we introduce some quantities. Let the quantity
\begin{align*}
K(u)=\int_{\R^N}|\nabla u|^2dx-\frac{Np+2b}{2(p+2)}\int_{\R^N}|x|^{-b}|u|^{p+2}dx,
\end{align*}
then it is known as the virial identity that for the solution $u$ of the equation \eqref{IDNLS},
\begin{align*}
\frac{d^2}{dt^2}\int|x|^2|u(t,x)|^2 dx=8K(u(t)).
\end{align*}

\subsection{Local Virial identity}\label{1261}

Let us start by introducing the local Virial identity
\begin{align*}
I(t)=\int\phi(x)|u(t,x)|^2 dx,
\end{align*}
then by direct computations (see for example \cite{Ardila-Cardoso-21}), we have that
\begin{lem}
 For any $\phi\in C^4(\R^N)$,
\begin{align*}
I'(t)=2Im\int \nabla\phi\cdot\nabla u \overline{u}dx,
\end{align*}
and
\begin{align*}
I''(t)=&4Re\sum\limits_{j,k=1}^N\int \partial_j\partial_k\phi\cdot \partial_ju \partial_k \overline{u}dx-\int|u|^2\Delta^2\phi dx\nonumber\\
&-\frac{2p}{p+2}\int |x|^{-b}|u|^{p+2}\Delta\phi dx\nonumber\\
&+\frac{4}{p+2}\int \nabla (|x|^{-b})\cdot\nabla\phi|u|^{p+2}dx.
\end{align*}
\end{lem}
From this lemma, if $\phi$ is radial, then one may find that
\begin{align}\label{2.1}
I'(t)=2Im\int\phi'\frac{x\cdot \nabla u}{r} \bar{u}dx,
\end{align}
and
\begin{align}\label{2.7937331}
I''(t)=&4\int \frac{\phi'}{r}|\nabla u|^2dx+4\int(\frac{\phi''}{r^2}-\frac{\phi'}{r^3})|x\cdot\nabla u|^2dx\nonumber\\
&-\frac{2p}{p+2}\int\Big[\phi''+(N-1+\frac{2b}{p})\frac{\phi'}{r}\Big]|x|^{-b}|u|^{p+2}dx
-\int\Delta^2\phi|u|^2dx,
\end{align}
here and in the sequel, $r$ denotes $|x|$.

We rewrite $I''(t)$ in (\ref{2.7937331}) as
\begin{align}\label{2.6}
I''(t)= 8K(u(t))+\mathcal{R}_1+\mathcal{R}_2+\mathcal{R}_3,
\end{align}
where
\begin{align*}
  \mathcal{R}_1=&4\int (\frac{\phi'}{r}-2)|\nabla u|^2dx+4\int(\frac{\phi''}{r^2}-\frac{\phi'}{r^3})|x\cdot\nabla u|^2dx,\\
\mathcal{R}_2=&-\frac{2p}{p+2}\int\Big[\phi''+(N-1+\frac{2b}{p})\frac{\phi'}{r}-2N-\frac{4b}{p}\Big]|x|^{-b}|u|^{p+2}dx,\\
\mathcal{R}_3=&-\int\Delta^2\phi|u|^2dx.
\end{align*}
We choose $\phi$ such that
\[
0\leq\phi\leq r^2,  \quad\phi''\leq2, \quad \phi^{(4)}\leq\frac{4}{R^2},
\]
and
\begin{align*}
   \phi=\left\{ \aligned
   r^2&, \ 0\leq r\leq R,\\
    0&,  \ r\ge 2R.
    \endaligned
  \right.
  \end{align*}
According to our needs, we choose the appropriate $R$ later.
 \section{Proof of Theorem \ref{thm:main1} and Theorem \ref{thm:main2}}\label{sec:proof}

In this section, we prove Theorem \ref{thm:main1} in subsection \ref{3.1} and prove Theorem \ref{thm:main2} in subsection \ref{3.2}.

Before this, we give a lemma about $K(u(t))$, see \cite{Ardila-Cardoso-21} for its proof.
\begin{lem}\label{lem3.222}
Let $\frac{4-2b}{N}<p<2^{\ast}_b$, then for all $t\in I$, there exits $\de_0>0$, such that
 \begin{align*}
K(u(t))<-\de_0\|\nabla u(t)\|_{L^2}^2.
\end{align*}
Moreover,
 \begin{align*}
\|\nabla u(t)\|_{L^2}\gtrsim 1.
\end{align*}
\end{lem}

\subsection{Proof of Theorem \ref{thm:main1}}\label{3.1}
Below we will introduce a lemma about $I''(t)$ when $\frac{4-2b}{N}<p\leq \frac4N$, which is similar to the result when $\frac{4}{N}<p< 2^{\ast}_b$. Although there is only one difference in the proof process, this is the essential reason for the different indicators in the two cases. These results play a vital role in proving Theorem \ref{thm:main1} and Theorem \ref{thm:main2}.

\begin{lem}\label{lem2.333}
 Let $\frac{4-2b}{N}<p< \min\{\frac{4}{N}, 2^{\ast}_b\}$ or $p =\frac{4}{N}$ with $0< b< \frac 4 N$, then there exists a constant $\delta_0>0$, such that
\begin{align*}
I''(t)\leq-4\delta_0\|\nabla u\|_{L^2}^2.
\end{align*}
\end{lem}
\begin{proof}
From \eqref{2.6}, we estimate the terms on the right hand side, term by term.
Firstly, notice that  when $0< b< \frac 4 N$, we have $p=\frac{4}{N}<2^{\ast}_b$. Hence, by Lemma \ref{lem3.222}, we get
\begin{align}\label{1.26}
K(u(t))<-\de_0\|\nabla u(t)\|_{L^2}^2.
\end{align}
Secondly, for the term $\mathcal{R}_1$, we claim that
\begin{align}\label{2.7}
\mathcal{R}_1\leq0.
\end{align}
Indeed, we divide the space $\R^N$ into two parts:
\[\{\frac{\phi''}{r^2}-\frac{\phi'}{r^3}\leq0\}\] and
\[ \{\frac{\phi''}{r^2}-\frac{\phi'}{r^3}>0\}.\]
When $\frac{\phi''}{r^2}-\frac{\phi'}{r^3}\leq0$, since $\phi'\leq 2r$, \eqref{2.7} is obvious.
When \[\frac{\phi''}{r^2}-\frac{\phi'}{r^3}>0,\]
then since $\phi''\leq2$, we have
\begin{align*}
\mathcal{R}_1\leq&4\int (\frac{\phi'}{r}-2)|\nabla u|^2dx+4\int(\frac{\phi''}{r^2}-\frac{\phi'}{r^3})|x|^2|\nabla u|^2dx\nonumber\\
=&4\int (\frac{\phi'}{r}-2)|\nabla u|^2dx+4\int(\phi''-\frac{\phi'}{r})\frac{|x|^2}{r^2}|\nabla u|^2dx\nonumber\\
=&4\int(\phi''-2)|\nabla u|^2dx\leq0.
\end{align*}
Next, we treat the term $\mathcal{R}_2$. Recall that
\[
\mathcal{R}_2=-\frac{2p}{p+2}\int\Big[\phi''+(N-1+\frac{2b}{p})\frac{\phi'}{r}-2N-\frac{4b}{p}\Big]|x|^{-b}|u|^{p+2}dx.
\]
If $0\leq r\leq R$, then
\[
\phi'=2r, \quad\phi''=2,
\]
and thus,
\[
\phi''+(N-1+\frac{2b}{p})\frac{\phi'}{r}-2N-\frac{4b}{p}=0.
\]
Hence,
\[
\supp\Big[\phi''+(N-1+\frac{2b}{p})\frac{\phi'}{r}-2N-\frac{4b}{p}\Big]\subset(R,\infty).
\]
By the Gagliardo-Nirenberg inequality we have
\begin{align}\label{1.23}
\mathcal{R}_2&\lesssim\int_{|x|>R}|x|^{-b}|u|^{p+2}dx\lesssim R^{-b}\|\nabla u\|_{L^2}^{2\tilde{\alpha}}\|u\|_{L^2}^{p+2-2\tilde{\alpha}}\nonumber\\
&\lesssim R^{-b}\|\nabla u\|_{L^2}^{2\tilde{\alpha}},
\end{align}
where $\tilde{\alpha}=\frac{Np}{4}\in(0,1]$.

Finally, for the term $\mathcal{R}_3$, we have
\begin{align}\label{2.9}
\mathcal{R}_3\lesssim R^{-2}\|u\|_{L^2}^2\lesssim  R^{-2}.
\end{align}
Thus, from \eqref{2.6}, \eqref{1.26}-\eqref{2.9}, we have
\begin{align*}
I''(t)\lesssim& -8\de_0\|\nabla u(t)\|_{L^2}^2+R^{-b}\|\nabla u(t)\|_{L^2}^{2\tilde{\alpha}}+R^{-2}.
\end{align*}
If $\tilde{\alpha}=1$, then the second term can be absorbed by the first one by fixing $R>0$ large enough, and thus we have
\begin{align*}
I''(t)\lesssim-4\de_0\|\nabla u(t)\|_{L^2}^2.
\end{align*}
If $0<\tilde{\alpha}<1$, by using Young's inequality for the second term, we can further obtain the above estimates.
The lemma is now proved.
\end{proof}

We are now in a position to prove Theorem \ref{thm:main1}.
\begin{proof}
Suppose that the maximal existence interval is $I=(-T_{*},T^*)$. We proceed by contradiction and assume that $T^*=+\infty$.

First, by the Lemma \ref{lem3.222}, there exists $C_0>0$ such that
\begin{align*}
 \|\nabla u(t)\|_{L^2}^2\geq C_0,
 \end{align*}
for all $t\in(-T_{*},T^*) $.
Get directly from the Lemma \ref{lem2.333}, we have
\begin{align*}
I''(t)\leq-4\de_0C_0.
\end{align*}
Further, integrate from $0$ to $t$ for the above formula,
\begin{align}
I'(t)\leq-4\de_0C_0t+I'(0).\label{1244}
\end{align}
We now may choose $T_0>0$ sufficiently large such that $I'(0)<2\de_0C_0T_0$. From this and \eqref{1244},
\begin{align}\label{2.11}
I'(t)\leq-2\de_0C_0t<0,\quad t\geq T_0.
\end{align}
By Lemma \ref{lem2.333} and \eqref{2.11}, we obtain
\begin{align}\label{2.12}
I'(t)=&\int_{T_0}^t I''(s) ds+I'(T_0)\nonumber\\
\leq&-4\de_0\int_{T_0}^t\|\nabla u(s)\|_{L^2}^2 ds,\quad t\geq T_0.
\end{align}
Moreover, from \eqref{2.1} and H\"older's inequality we deduce
\begin{align}\label{2.13}
|I'(t)|=&|2Im\int\phi'\frac{x\cdot \nabla u}{r} \bar{u}dx|\nonumber\\
\lesssim&R\|\nabla u(t)\|_{L^2}\|u(t)\|_{L^2}\nonumber\\
\lesssim&R\|\nabla u(t)\|_{L^2}.
\end{align}
From \eqref{2.11}-\eqref{2.13},  we can obtain
\begin{align*}
\int_{T_0}^t\|\nabla u(s)\|_{L^2}^2 ds\lesssim |I'(t)|\lesssim\|\nabla u(t)\|_{L^2}.
\end{align*}
Define $f(t):=\int_{T_0}^t\|\nabla u(s)\|_{L^2}^2 ds$ and thus we have $f^2(t)\lesssim f'(t)$.
Finally, taking $T'>T_0$ and integrating on $[T',t)$ we get
\[
t-T'\lesssim \int_{T'}^t\frac{f'(s)}{f^2(s)}ds=\frac{1}{f(T')}-\frac{1}{f(t)}\leq\frac{1}{f(T')}.
\]
Noting that ${f(T')}>0$, letting $t\rightarrow \infty$ we arrive to a contradiction. Hence Theorem \ref{thm:main1} is completed.
\end{proof}

\subsection{Proof of Theorem \ref{thm:main2}}\label{3.2}

Next, we give the proof of Theorem \ref{thm:main2}.
Before that, we will also give an important lemma. Here we can see the difference between two cases in the treatment of $\mathcal{R}_2$.

\begin{lem}\label{2.554623}
Let $\frac{4}{N}<p< 2^{\ast}_b$, then there exists $\alpha>1$, such that
\begin{align*}
I''(t)\lesssim -8\de_0 \|\nabla u(t)\|_{L^2}^2+R^{-b}\|\nabla u(t)\|_{L^2}^{2\alpha}+R^{-2}.
\end{align*}
\begin{proof}
The estimates for $K(u(t))$, $\mathcal{R}_1$ and $\mathcal{R}_3$ are the same as in Lemma \ref{lem3.222} and Lemma \ref{lem2.333}, it reduces to estimate $\mathcal{R}_2$.
By the Gagliardo-Nirenberg inequality we have
\begin{align*}
\mathcal{R}_2\lesssim\int_{|x|>R}|x|^{-b}|u|^{p+2}dx\lesssim& R^{-b}\|\nabla u(t)\|_{L^2}^{2\alpha}\|u\|_{L^2}^{p+2-2\alpha}\nonumber\\
\lesssim& R^{-b}\|\nabla u(t)\|_{L^2}^{2\alpha},
\end{align*}
where $\alpha=\frac{Np}{4}\in (1, \frac{2^{\ast}_0 N}{4})$.\\
Combing \eqref{2.6}, \eqref{2.7}, \eqref{2.9}, Lemma \ref{lem3.222} with the last inequality, we get the desired result.
\end{proof}
\end{lem}

Now, we give the proof of Theorem \ref{thm:main2}.
\begin{proof}
Fixing $T>0$, we let
\begin{align*}
R(T)=(3\de_0)^{-\frac 1 b}\sup_{t\in [0,T]}\|\nabla u(t)\|_{L^2}^{\frac{2\alpha-2}{b}}.
\end{align*}
Then, we have
\begin{align}\label{2.16}
R(T)^{-b}\|\nabla u(t)\|_{L^2}^{2\alpha}\leq 3\de_0\|\nabla u(t)\|_{L^2}^{2-2\alpha}\|\nabla u(t)\|_{L^2}^{2\alpha}=3\de_0\|\nabla u(t)\|_{L^2}^{2}
\end{align}
and
\begin{align}\label{2.17}
R(T)^{-2}\lesssim 3\de_0\|\nabla u(t)\|_{L^2}^{2}.
\end{align}
Combing \eqref{2.16}, \eqref{2.17} and Lemma \ref{2.554623}, we obtain
\begin{align*}
  I''(t)\lesssim-\delta_0\|\nabla u(t)\|_{L^2}^2.
\end{align*}
Applying the classical analysis identity
\begin{align*}
  I(T)=I(0)+I'(0)\cdot T+\int_0^T\int_0^s I''(t)dtds.
\end{align*}
From the definition of $I(t)$ and the choosing of $\phi$, we have
\begin{align*}
 I(T)\geq0,
\end{align*}
and
\begin{align*}
I(0)\lesssim R(T)^2.
\end{align*}
By \eqref{2.13}, we get
\begin{align*}
I'(0)\lesssim R(T)\|\nabla u_0\|_{L^2}\lesssim R(T).
\end{align*}
Collecting the estimates above we have
\begin{align}\label{2.18}
  \delta_0\int_0^T\int_0^s\|\nabla u(t)\|_{L^2}^2 dt ds
&\lesssim R(T)^2+R(T)\cdot T.
\end{align}
Recalling that $\|\nabla u(t)\|_{L^2}^2\geq C_0>0$, for all $t\in [0,T)$, we get
\begin{align}\label{2.19}
\delta_0\int_0^T\int_0^s\|\nabla u(t)\|_{L^2}^2 dt ds\gtrsim T^2.
\end{align}
Further, by \eqref{2.18}, \eqref{2.19} and the elementary inequality we have
\begin{align*}
T^2\lesssim R(T)^2+R(T)\cdot T\leq R(T)^2+\frac{R(T)^2+T^2}{2}.
\end{align*}
From this, we can obtain
\begin{align*}
R(T)^2\gtrsim T^2.
\end{align*}
Noting $R(T)>0$, hence
\begin{align*}
  R(T)\gtrsim T.
\end{align*}
That is
\begin{align*}
  \sup_{t\in[0,T]}\|\nabla u(t)\|_{L^2}\gtrsim T^{\frac{b}{2\alpha-2}}.
\end{align*}
This proves Theorem \ref{thm:main2}.
\end{proof}

In conclusion, we conclude the proofs of Theorem \ref{thm:main1} and Theorem \ref{thm:main2}.

\section*{Acknowledgements}

The authors are very grateful to the anonymous referees for helpful comments and suggestions.

\end{document}